\documentclass[A4, 11pt, reqno]{amsart}
\usepackage{tikz-cd}
\usetikzlibrary{arrows.meta,positioning,calc}
\newcommand{\ColumnSep}{3.6cm}
\newcommand{\RowSep}{3.2cm}
\newcommand{\LoopReach}{6.5mm}
\newcommand{\LoopFlare}{4.5mm}
\newcommand{\LoopMouth}{1.4mm}
\pgfdeclarelayer{arrows}
\pgfsetlayers{arrows,main}
\usepackage{quiver}
\usepackage[T1]{fontenc}
\usepackage{amsmath,amssymb,amsthm}
\usepackage[backend=biber]{biblatex} 
\usepackage{hyperref}
\theoremstyle{definition}

\newtheorem{theorem}{Theorem}[section]
\newtheorem{proposition}[theorem]{Proposition}

\newtheorem{definition}[theorem]{Definition}
\newtheorem{example}[theorem]{Example}

\newtheorem{question}[theorem]{Question}
\begin{document}
\title{Tensor Product Does Not Behave Additively on Delooping Level}
\author{Ryan Lam}
\address{School of Mathematics, Fry Building, Woodland Road, Bristol, BS8 1UG, United Kingdom}
\email{ryan.lam@bristol.ac.uk}

\begin{abstract}
  The finitistic dimension conjecture is one of the most important homological conjectures in the representation theory of finite-dimensional algebras. Recently, G\'elinas \cite{GELINAS2022108052} proposed the notion of ``delooping level'' to study the finitistic dimension Conjecture. This article reports a misconception between the delooping level and the tensor product of algebras. Namely, for finite-dimensional algebras $A$ and $B$, the difference $\operatorname{dell}(A\otimes_{\mathbf{k}}B) - \left(\operatorname{dell}(A)+ \operatorname{dell}(B)\right)$ can be arbitrarily large or negative.
\end{abstract}
\maketitle

We set up our notations first. Let $A$ be an Artin algebra, or more specifically, a finite-dimensional $\mathbf k$-algebra over an algebraically closed field $\mathbf k$. Unless otherwise specified, all modules are right modules. Denote $\operatorname{Mod} A$ to be the category of all right $A$-modules, and $\operatorname{mod} A$ to be the category of all finitely generated right $A$-modules. Also denote the finitely generated projectively stable right module category by $\underline{\bmod A}$. 
\section{Introduction}
In 1960, Bass \cite{bass_finitistic_1960} investigated the following two invariants, introduced by Auslander and Buchsbaum in \cite{aus_busch}:
\begin{align*}
    \operatorname{findim} A &:= \sup\{\operatorname{pd}(M) : \operatorname{pd}(M)<\infty, M \in\operatorname{mod} A \} \\
    \operatorname{FinDim} A &:= \sup\{\operatorname{pd}(M) : \operatorname{pd}(M)<\infty, M \in\operatorname{Mod} A \}
\end{align*}
called the \textbf{little} and \textbf{big finitistic dimension}, and conjectured that they are equal and finite for all finite-dimensional algebras $A$.  About 30 years later, Huisgen-Zimmermann \cite{huisgen-zimmermann_homological_1992} was the first to show that they can be different. A few years later, Smal{\o} \cite{Smal98} showed they are arbitrarily different, even for algebras with radical cubed zero. However, whether $\operatorname{findim} A < \infty$ or $ \operatorname {FinDim} A < \infty$ for an arbitrary algebra $A$ remains open, referred to as the little and big \textit{finitistic dimension conjecture}, respectively. Early development of the conjecture was documented in a survey paper by Huisgen-Zimmermann \cite{Huisgen1995}.

Recently, there has been a search for conditions that would imply the big finitistic dimension conjecture. G\'elinas  introduced ``delooping level'' in \cite{GELINAS2022108052}. Let $\Omega: \underline{\bmod A} \to \underline{\bmod A}$ be the syzygy functor. In this article, if we are taking a syzygy in $\operatorname{mod} A$, then we are always taking the kernel of the projective cover. In his paper, G\'elinas showed $\operatorname{dell}(A)\geq \operatorname{FinDim}(A^{op})$ for all algebras $A$. 

\begin{definition}
      The \textbf{delooping level} of an $A$-module $M$ is the least integer $n$, if exists, such that there exists an $A$-module $N$ for which $\Omega^n(M)$ is a stable summand of $\Omega^{n+1}(N)$. We write $\operatorname{dell}_A(M) = n$ in this case. If such $n$ does not exist, we write $\operatorname{dell}_A(M) = \infty$.

      The \textbf{delooping level} of an algebra $A$ is defined by 
    \[ \operatorname{dell}(A) := \sup_{M \text{ is a simple module}} \operatorname{dell}_A(M).\]
\end{definition}

There seems to be an understanding that delooping level and finitistic dimension behave additively under tensor products, and this was commonly attributed to Rickard. For example, it is mentioned in an article \cite[p.217, l2]{barrios_delooping_2026} by Barrios, Lanzilotta and Mata. While the result for finitistic dimensions is true and, in fact, recorded by Eilenberg, Rosenberg and Zelinsky \cite{Eilenberg_Rosenberg_Zelinsky_1957} and Smal{\o} \cite{Smal98}, we show that the result for delooping levels is false. There are two ways to interpret the phrase ``delooping level behaves additively on tensor products''; we prove both wrong in the following theorem.
\begin{theorem}
 There is algebra $A$ and simple $A$-modules $M, N$ such that
\begin{equation}\operatorname{dell}_A(M)+\operatorname{dell}_A(N) \neq \operatorname{dell}_{A\otimes_\mathbf{k} A}(M \otimes_\mathbf{k} N)
    \label{eq:count_1}
    \end{equation}
    In particular, this implies there are algebras $A$ and $B$ such that
    \[\operatorname{dell}(A\otimes_{\mathbf{k}}B) \neq \operatorname{dell}(A)+ \operatorname{dell}(B).\]
\label{thm:dell_count}
\end{theorem}
Motivated by theorem \ref{thm:dell_count}, for arbitrary algebra $R$ we define
\[s(R) := \operatorname{dell}(R\otimes_\mathbf{k} R) -2\operatorname{dell}(R).\] 
One might ask whether there is a bound on $s(R)$. We answer this in the negative.
\begin{theorem}
    Let $A$ be the algebra in Theorem \ref{thm:dell_count}. If we inductively define a family of algebras $\{A_n\}_{n \geq 2}$ by
    \[A_2 = A, \qquad A_n = \begin{pmatrix}
    A_{n-1} & M \\ 0 &\mathbf{k}
\end{pmatrix}\]
where $M$ is a specified simple left $A_{n-1}$-module (see Section \ref{sec:gen}), then we have $s(A_n) \geq n-1$ for all $n \geq 2$.
\label{thm:dell_gen}
\end{theorem}
The key in establishing these results is to carefully analyze the support of certain modules. Recall that a module $M$ over a quiver algebra is \textbf{supported} at vertex $v$ if, in its quiver representation, the vertex $v$ is assigned to a non-zero vector space. We write $v \in \operatorname{supp}(M)$ in this case.

The remainder of the paper will be organized as follows: Sections \ref{sec:counterexample} and \ref{sec:gen} will be devoted to proving Theorems \ref{thm:dell_count} and \ref{thm:dell_gen} respectively. In fact, Theorem \ref{thm:dell_gen} is a strict generalization of Theorem \ref{thm:dell_count} in the sense that substituting $n=2$ in the proof of Theorem \ref{thm:dell_gen} will recover the exact proof for Theorem \ref{thm:dell_count}. But we still separate them to make the important observations and constructions more obvious. Finally, in Section \ref{sec:last}, we will present further generalizations and open questions for further studies. 

\textbf{Acknowledgement.} The counterexample establishing Theorems \ref{thm:dell_count} and \ref{thm:dell_gen} was found by the newest publicly available (Pro Subscription) model of ChatGPT at the time of writing (\today). The exposition has been fully reviewed and rewritten by the author using his own ideas, and the author takes full responsibility for any errors and typographical errors made in this paper. Figure \ref{fig:A_tensor_A} is originally hand-drawn but converted to TikZ using the same AI and was further refurbished by the author. The GAP programming language \cite{GAP4} and the package \texttt{QPA} \cite{qpa} were used to check the validity of calculations in Section \ref{sec:counterexample}. The author would like to thank his supervisor, Jeremy Rickard, for pointing out several functions in \texttt{QPA} that helped the verification.
\section{The Counterexample}
\label{sec:counterexample}
 Before introducing the counterexample for Theorem \ref{thm:dell_count}, we introduce two important propositions to calculate delooping levels. The propositions are proved by G\'elinas \cite{GELINAS2022108052}.

\begin{proposition}
    Given an $A$-module $M$, positive integer $k$:
  $\operatorname{dell}_A(M) \leq n$ if and only if $\Omega^n(M)$ is a stable retract of $\Omega^{n+1}\Sigma^{n+1}\Omega^n(M)$, where $\Sigma:\underline{\bmod A} \to \underline{\bmod A}$ is the left adjoint to the syzygy functor. In particular,  $\operatorname{dell}_A(M) = 0$ if and only if $M$ is \textit{torsionless}, that is, $M$ is a submodule of a projective module.
    \label{prop:useful_facts}
\end{proposition}

To compute $\Sigma M$ for an $A$-module $M$, one can take a left projective approximation $f: M\to P$ of $M$ and take its cokernel. Recall that $f$ is a left projective approximation if for any $g: M \to Q$ where $Q$ is a projective $A$-module, there is $h: P \to Q$ such that $hf = g$. The GAP package \texttt{QPA} also supports $\Sigma^n$ of a module \texttt{M} by calling \texttt{Transpose(NthSyzygy(Transpose(M),n))}.

Now we present our counterexample to Theorem \ref{thm:dell_count}. Let $A$ be an 8-dimensional quiver algebra $kQ/I$ defined by the following quiver $Q$:
\[ \begin{tikzpicture}
\node (v1) at (0,0) {$1$};
\node (v4) at (0,-2) {$4$};
\node (v2) at (1.5,-1) {$2$};
\node (v3) at (3,-1) {$3$};

\draw[-stealth] (v1) -- node[above]{$\alpha$} (v2) ;
\draw[-stealth]  (v2) -- node[above]{$\beta$}  (v3);
\draw[-stealth]  (v4) -- node [below]{$\gamma$} (v2);
\path[-stealth] (v1) edge [loop left] node {$\delta$}() ;
\end{tikzpicture}\]
with $I = \langle\alpha\beta, \gamma\beta, \delta\alpha, \delta^2\rangle=(\operatorname{rad} kQ)^2$. Note that path multiplication goes from left to right. We will use $S_i$ to denote the simple module supported at vertex $i$ from now on. Now we are ready to prove Theorem \ref{thm:dell_count}.
\begin{proof}
To find the delooping level of $A$, we notice that the Loewy structures of the four indecomposable projective modules of $A$, denoted by $P_1, P_2, P_3, P_4$, can be written as follows.
\[\begin{matrix}&1& \\1 &&2\end{matrix}, \quad\begin{matrix}
    2 \\3
\end{matrix},\quad 3,\quad \begin{matrix}
    4 \\2
\end{matrix}\]
, where $i$ is used as an abbreviation for $S_i$. Hence $S_1, S_2, S_3$ are submodules of projectives, and hence their delooping level is 0. $S_4$ is not torsionless, but observe that $\Omega(S_4) = S_2$ and $\Omega^2\binom{1}{2} = S_1\oplus S_2$, hence $\Omega(S_4)$ is a (stable) summand of $\Omega^2\binom{1}{2}$, hence \[\operatorname{dell}_A(S_4) = \operatorname{dell}(A) = 1.\]

Now we move on to analyse $A\otimes_{\mathbf k} A$. It suffices to show $\operatorname{dell}_{A\otimes_{\mathbf{k}} A}(S_4\otimes_{\mathbf{k}} S_4) > 2$.  $A\otimes_{\mathbf k} A$ is a $8\times 8 = 64$-dimension algebra. The simple modules of $A \otimes_{\mathbf{k}} A$ are $S_i \otimes_{\mathbf{k}} S_j$, and will be denoted by by $S_{ij}$. We can realize it as a quiver algebra $k\widetilde{Q}/\widetilde{I}$ as follows. The underlying quiver $\widetilde{Q}$ is the Cartesian product of two copies of $Q$. This is shown in Figure \ref{fig:A_tensor_A}. $\widetilde{Q}$ has 16 vertices, labelled by $(i,j)$ where $1 \leq i, j \leq 4$. Fixing $i$, there are arrows $\alpha^i, \beta^i, \gamma^i, \delta^i$ (labelled by \textcolor{red}{red}). Fixing $j$, we have arrows $\alpha_j, \beta_j, \gamma_j, \delta_j$ (labelled by black), the total of $32$ arrows contains a copy of $Q$ in each fixed ``coordinate''. 
\begin{figure}[h!]
    \centering
    \begin{tikzpicture}[
  x=\ColumnSep, y=\RowSep,
  vertex/.style={
    draw=black, fill=white, rounded corners=1pt,
    line width=0.55pt, minimum width=7.5mm,
    minimum height=5.8mm, inner sep=1pt,
    outer sep=0.3pt
  },
  arrow/.style={
    -{Stealth[length=2mm,width=1.35mm]},
    line width=0.7pt, line cap=round,
    shorten <=0.4pt, shorten >=0.5pt
  },
  crossing/.style={
    preaction={draw=white,arrows=-,line width=2.7pt}
  },
  black arrow/.style={arrow,draw=black},
  red arrow/.style={arrow,draw=red,crossing},
  black label/.style={
    font=\normalsize,text=black,fill=white,inner sep=0.8pt
  },
  red label/.style={black label,text=red}, scale = 0.9
]

\foreach \id/\x/\y in {
  11/0/0,       12/0/1,       13/0/2,
  21/1/0,       22/1/1,       23/1/2,
  31/3/0,       32/3/1,       33/3/2,
  41/2/0.5,     42/2/1,       43/2/1.5,
  14/-0.5/1.5,  24/0.5/1.5,   34/1.5/1.5,
  44/0.5/0.5
}{
  \node[vertex] (v\id) at (\x,\y) {$\id$};
}

\begin{pgfonlayer}{arrows}
\draw[black arrow] (v14) --
  node[black label,pos=0.22,above=2pt] {$\alpha_4$} (v24);
\draw[black arrow] (v24) --
  node[black label,pos=0.24,above=2pt] {$\beta_4$} (v34);
\draw[black arrow] (v11) --
  node[black label,above=2pt] {$\alpha_1$} (v21);
\draw[black arrow] (v12) --
  node[black label,pos=0.28,above=2pt] {$\alpha_2$} (v22);
\draw[black arrow] (v13) --
  node[black label,above=2pt] {$\alpha_3$} (v23);
\draw[black arrow] (v21) --
  node[black label,pos=0.73,above=2pt] {$\beta_1$} (v31);
\draw[black arrow] (v41) --
  node[black label,above left=1pt] {$\gamma_1$} (v21);
\draw[black arrow] (v42) --
  node[black label,above=2pt] {$\gamma_2$} (v22);
\draw[black arrow] (v22) to[bend right = 30]
  node[black label, pos=0.8] {$\beta_2$} (v32);
\draw[black arrow] (v43) --
  node[black label,above right=1pt] {$\gamma_3$} (v23);
\draw[black arrow] (v23) --
  node[black label,above=2pt] {$\beta_3$} (v33);
\draw[black arrow,crossing] (v44) --
  node[black label,pos=0.2,left=2pt] {$\gamma_4$} (v24);
\draw[red arrow] (v11) --
  node[red label,left=2pt] {$\alpha^1$} (v12);
\draw[red arrow] (v21) --
  node[red label,pos=0.4,right=2pt] {$\alpha^2$} (v22);
\draw[red arrow] (v31) --
  node[red label,right=2pt] {$\alpha^3$} (v32);
\draw[red arrow] (v41) --
  node[red label,right=2pt] {$\alpha^4$} (v42);
\draw[red arrow] (v12) --
  node[red label,pos=0.76,left=2pt] {$\beta^1$} (v13);
\draw[red arrow] (v22) --
  node[red label,pos=0.76,right=2pt] {$\beta^2$} (v23);
\draw[red arrow] (v32) --
  node[red label,right=2pt] {$\beta^3$} (v33);
\draw[red arrow] (v42) --
  node[red label,right=2pt, pos = 0.4] {$\beta^4$} (v43);
\foreach \from/\to/\i in {14/12/1,24/22/2,34/32/3,44/42/4}{
  \draw[red arrow] (v\from) --
    node[red label,above right=1pt] {$\gamma^{\i}$} (v\to);
}
\foreach \v/\i in {11/1,12/2,13/3,14/4}{
  \draw[black arrow]
    ([yshift=\LoopMouth]v\v.west)
    .. controls
       ([xshift=-\LoopReach,yshift=\LoopFlare]v\v.west)
       and
       ([xshift=-\LoopReach,yshift=-\LoopFlare]v\v.west)
    .. node[black label,pos=0.5,left=2pt] {$\delta_{\i}$}
       ([yshift=-\LoopMouth]v\v.west);
}

\foreach \v/\i in {11/1,21/2,31/3,41/4}{
  \draw[arrow,draw=red]
    ([xshift=-\LoopMouth]v\v.south)
    .. controls
       ([xshift=-\LoopFlare,yshift=-\LoopReach]v\v.south)
       and
       ([xshift=\LoopFlare,yshift=-\LoopReach]v\v.south)
    .. node[red label,pos=0.5,below=2pt] {$\delta^{\i}$}
       ([xshift=\LoopMouth]v\v.south);
}

\end{pgfonlayer}
\end{tikzpicture}
    \caption{A figure of the quiver $\widetilde{Q}$. We use the shorthand $ij$ to denote the vertex $(i, j).$}
    \label{fig:A_tensor_A}
\end{figure}
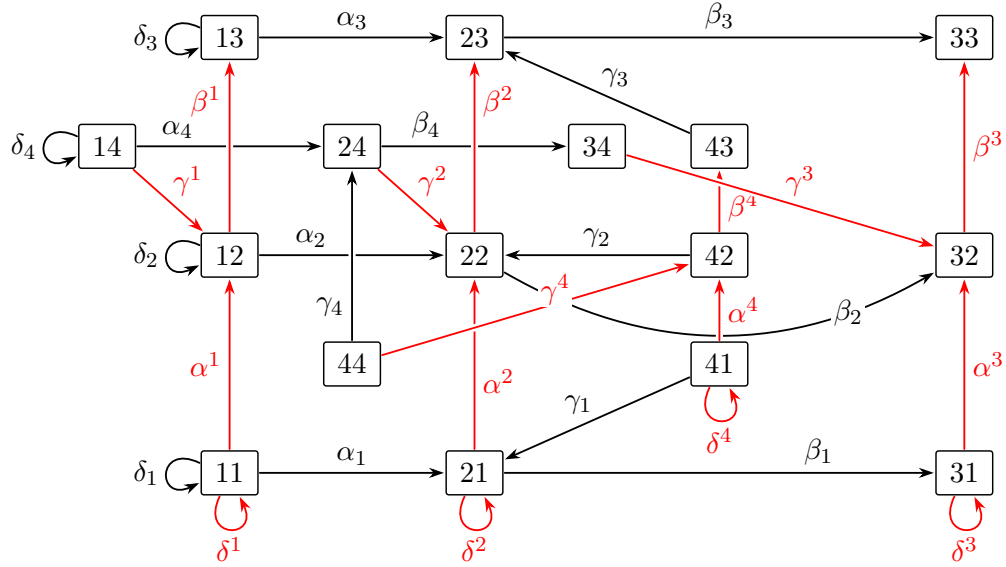

The ideal $\widetilde{I}$ is generated by all the length-two paths of the same color and any commutative square of the arrows consisting of one color after another. In particular, note that $\operatorname{rad}^3 A=0$ in this case, because any non-zero length-two path can be identified to end with the arrow of any color by the commutative requirement, and hence composing with another arrow with any color will annihilate the element.

We only need to compute $\Omega^3\Sigma^3\Omega^2(S_{44})$\footnote{Note that we will suppress the difference between taking syzygies in $A$ and $A \otimes_{\mathbf{k}} A$ and denote both functors by $\Omega$. The reader should be able to distinguish them by context.} and check that $\Omega^2(S_{44})$ is a stable summand of it. One can directly calculate it by hand, but we outline a few important features in this calculation below.

Computing the first three terms of the left projective approximation gives: 
\begin{equation}\Omega^2(S_{44}) = \begin{matrix}
    43 &&22&&34\\&23&&32
\end{matrix} \to P_{11}\oplus P_{42} \oplus P_{24} \to P_{44}\oplus P_{11}^2\stackrel{q}\to P_{11}^m\to \cdots.
\label{eq:LProj_approx}
\end{equation}
In particular, note that $\Sigma^3\Omega^2(S_{44})=\Sigma^3\left(\begin{matrix}
    43 &&22&&34\\&23&&32
\end{matrix}\right) = \operatorname{coker} q$ is only supported on the 4 vertices $11, 12, 21, 22$. Moreover, the third syzygy module $\Omega^3\Sigma^3\Omega^2(S_{44})$ would \textbf{not} be supported on $14, 24, 34, 44, 41, 42, 43, 44$  Hence $\Omega^2(S_{44})$ is not a stable retract of $\Omega^3\Sigma^3\Omega^2(S_{44})$, hence $\operatorname{dell}(S_{44})>2$, and the theorem is proved.
\end{proof}
In fact, the delooping level of $A \otimes_{\mathbf{k}}A$ is $3$. This can be checked by \texttt{QPA}.
\section{Generalization of the Counterexample}
\label{sec:gen}
In the proof of Theorem \ref{thm:dell_count}, it is crucial that the support of $U:= \Sigma^3\Omega^2(S_{44})$ is supported in the 4 vertices $11, 12, 21, 22,$ hence the syzygies of $U$ would not be supported on the vertex $44$. We generalize this observation to arbitrary quiver algebras.
\begin{proposition}
    Let $A = kQ/I$ be a quiver algebra. Let $\mathcal{V} \subseteq Q_0$ be a successor-closed subset of vertices. That is, any arrows that start in $\mathcal{V}$ must end in $\mathcal{V}$. If $\operatorname{supp}(M) \subseteq \mathcal{V}$, then $\operatorname{supp}(\Omega^t(M))  \subseteq \mathcal{V}$ for any non-negative integer $t$.
    \label{prop:succ_syz}
\end{proposition}
\begin{proof}
    The projective cover of $M$, $P_M$ would also be supported at $\mathcal{V}$ since $\mathcal{V}$ is successor-closed; hence $\Omega(M)$ is supported at $\mathcal{V}$. Then one can induct.
\end{proof}
\begin{example}
    In our counterexample $A$ in Section \ref{sec:counterexample}, $U:= \Sigma^3\Omega^2(S_{44})$ is supported on the successor-closed set:
    \[\mathcal{V}_2 := \{(i, j): 1\leq i \leq 3, 1\leq j \leq 3\}.\]
    Hence so is $\Omega^3(U)$, which lead us to conclude that $\operatorname{dell}(S_{44}) >2.$
\end{example}
This can be readily generalized. For integer $n \geq 2$, define a family of quiver algebras $A_n := kQ_n/I_n$ by the following quiver $Q_n$:
\[
\begin{tikzpicture}
\node (v1) at (0,0) {$1$};
\node (v4) at (0,-2) {$n+2$};
\node (v2) at (1.5,-1) {$2$};
\node (v3) at (3,-1) {$3$};
\node (vd) at (4.5, -1) {$\cdots$};
\node (vn) at (6.5, -1) {$n+1$};

\draw[-stealth] (v1) -- node[above]{$\alpha$} (v2) ;
\draw[-stealth]  (v2) -- node[above]{$\beta_1$}  (v3);
\draw[-stealth]  (v3) -- node[above]{$\beta_2$}  (vd);
\draw[-stealth]  (vd) -- node[above]{$\beta_{n-1}$}  (vn);
\draw[-stealth]  (v4) -- node [below]{$\gamma$} (v2);

\path[-stealth] (v1) edge [loop left] node {$\delta$}() ;

\end{tikzpicture} 
\]
with $I_n = (\operatorname{rad} kQ_n)^2$. This is consistent with our description in Theorem \ref{thm:dell_gen}. That is, recall we defined a family of algebras $\{A_n\}_{n \geq 2}$ by
    \[A_2 = A, \qquad A_n = \begin{pmatrix}
    A_{n-1} & M \\ 0 &\mathbf{k}
\end{pmatrix}\]
where $M$ will now be the simple left $A_{n-1}$-module associated to the vertex $n$. We are now ready to provide the proof of Theorem \ref{thm:dell_gen}.
\begin{proof}
     The indecomposable projective modules of $A_n$ are: \[ \begin{matrix}&1& \\1 &&2\end{matrix}, \qquad\begin{matrix}
    2 \\3
\end{matrix}, \qquad\begin{matrix}
    3 \\4
\end{matrix}, \qquad \cdots \begin{matrix}
    n \\(n+1)
\end{matrix},\qquad (n+1),\qquad \begin{matrix}
    (n+2) \\2
\end{matrix}
\]
    In particular, the simple modules $S_1, S_2, \cdots, S_{n+1}$ are torsionless modules. And similar to Theorem \ref{thm:dell_count}, $\Omega(S_{n+2})$ is a (stable) summand of $\Omega^2\binom{1}{2}$, hence $\operatorname{dell}(A_n) = \operatorname{dell}_A(S_{n+2})= 1$.

    We now analyze $A_n \otimes_{\mathbf{k}} A_n $ as a quiver algebra $ \mathbf{k}\widetilde{Q}_n/\widetilde{I_n}$. Denote the simple module $S_{i}\otimes_{\mathbf{k}} S_{j}$ by $S_{(i, j)}$. Similar to Theorem \ref{thm:dell_count}, we claim that the simple $S := S_{(n+2, n+2)}$ has delooping level $n+1$. It suffices to show that $\Sigma^{n+1}\Omega^{n}(S)$ is supported in vertices ${(1,1)}, {(1, 2)}, {(2, 1)}, {(2, 2)}$. If this is true, then by Proposition \ref{prop:succ_syz},  the successor-closed set
    \[\mathcal{V}_n := \{(i, j) : 1\leq i \leq n+1, 1\leq j\leq n+1\}\subseteq\widetilde{Q}_0\]
    would imply $\Omega^{n+1}\Sigma^{n+1}\Omega^{n}(S)$ was not supported at the vertices $(n+2, 2)$, but it is not difficult to see that $\Omega^n(S)$ is supported at a vertex with label $n+2$ either at first or second coordinate. Hence $\Omega^n(S)$ would not be a stable summand of $\Omega^{n+1}\Sigma^{n+1}\Omega^{n}(S)$, 
    therefore $\operatorname{dell}_{A\otimes_{\mathbf{k}}
    A}(S) > n$ as required.

    We now compute $\Sigma^{n+1}\Omega^{n}(S)$. First note that in $A$, there is a projective resolution:
    \[0 \to P_{n+1}\to P_n \to \cdots \to P_2\to P_{n+2}\to S_{n+2}\to 0.\]
   Computing Ext groups, when $i\geq 1$, we have:
    \[\operatorname{Ext}^i_A(S_{n+2}, A) = \begin{cases}
        \mathbf{k} &\text{ if } i=1, n\\  0 &\text{otherwise}
    \end{cases}.\]
    We now use the K\"unneth formula for Ext groups. We use one of the versions formulated by Gao, Zhang and Zhu \cite[Theorem A.1]{gao_weakly_2026}. 
    \[\operatorname{Ext}^k_{A\otimes_{\mathbf{k}} A }(S, A\otimes_{\mathbf{k}} A ) = \bigoplus_{i, j:i+j=k}\operatorname{Ext}^i_A(S_{n+2}, A)\otimes_{\mathbf{k}} \operatorname{Ext}^j_A(S_{n+2}, A),\]
    In particular, \begin{equation}
        \operatorname{Ext}^1_{A\otimes_{\mathbf{k}} A }\left(\Omega^{k-1}(S), A\otimes_{\mathbf{k}} A\right) =\operatorname{Ext}^k_{A\otimes_{\mathbf{k}} A }\left(S, A\otimes_{\mathbf{k}} A\right) = 0
        \label{eq:ext_vanish}
    \end{equation}  for $3 \leq k \leq n$.
    Now, let $Q$ be a projective $A\otimes_{\mathbf{k}} A$-module. Consider the short exact sequence \[0 \to \Omega^k(S) \to P\to \Omega^{k-1}(S)\to 0\] and applying the functor $\operatorname{Hom}_{A\otimes_{\mathbf{k}} A}(-, Q)$ one has that \[\operatorname{Hom}_{A\otimes_{\mathbf{k}} A}(P, Q) \to \operatorname{Hom}_{A\otimes_{\mathbf{k}} A}(\Omega^k(S), Q) \quad \]
is surjective due to the vanishing of $\operatorname{Ext}^1$ in \eqref{eq:ext_vanish}. Hence $\Omega^k(S)  \to P$ is a left projective approximation, hence $\Sigma\Omega^k(S) =\Omega^{k-1}(S)$. Repetitive application of the formula gives:
\[\Sigma^{n+1}\Omega^n(S) = \Sigma^{n}\Omega^{n-1}(S) = \Sigma^{n-1}\Omega^{n-2}(S) = \cdots = \Sigma^3\Omega^2(S).\]
Now we can compute the indecomposable projectives of $A \otimes_{\mathbf{k}} A$ and verify that:
\[\Omega^2(S) = \begin{matrix}
    (n+2, 3) &&
    (2,2)&&(3,n+2)\\&(2,3)&&(3,2)
\end{matrix}.\]
Then the remainder of the computation is exactly the same as \eqref{eq:LProj_approx} in Theorem \ref{thm:dell_count}. That is, we have a left projective approximation 
\[\Omega^2(S)  \to P_{(1,1)}\oplus P_{(n+2,2)} \oplus P_{(2,n+2)} \to P_{(n+2,n+2)}\oplus P_{11}^2\stackrel{q_n}\to P_{11}^m\to \cdots.\]
Hence $\operatorname{coker}(q_n) = \Sigma^3\Omega^2(S) = \Sigma^{n+1}\Omega^n(S)$ is still supported in vertices ${(1,1)}, {(1, 2)}, {(2, 1)}, {(2, 2)}$ as required.
\end{proof}
\section{Future Work and Open Questions}
\label{sec:last}
In this section, we will look at modifications of Theorem \ref{thm:dell_gen} and pose them as open questions. Firstly, note that the following example shows that $s(R)$ can be negative too.
\begin{example}
    Consider the following quiver algebra $A' := kQ'/I'$ with the quiver $Q'$: \[
\begin{tikzpicture}
\node (v1) at (0,-1) {$1$};
\node (v4) at (1.5,-2.5) {$4$};
\node (v2) at (1.5,-1) {$2$};
\node (v2r) at (1.3, -1) {};
\node (v4r) at (1.3, -2.5) {};
\node (v2l) at (1.7, -1) {};
\node (v4l) at (1.7, -2.5) {};
\node (v3) at (3,-1) {$3$};

\draw[-stealth] (v1) -- node[above]{$\alpha$} (v2) ;
\draw[stealth-]  (v2) -- node[above]{$\beta$}  (v3);
\draw[stealth-]  (v4l) -- node [right]{$\gamma$} (v2l);
\draw[stealth-]  (v4r) -- node [left]{$\delta$} (v2r);
\path[-stealth] (v1) edge [loop left] node {$\ell$}() ;

\end{tikzpicture}\]
with $I' = \langle\ell^2, \ell\alpha, \alpha\gamma, \alpha\delta, \beta \delta\rangle$, then one can compute $\operatorname{dell}_{A'}(S_3) = \operatorname{dell}(A') = 2$ but $\operatorname{dell}_{A'  \otimes_{\mathbf{k}} A'}(S_3 \otimes_{\mathbf{k}} S_3) =  \operatorname{dell}(A'\otimes_{\mathbf{k}}A')=3.$ 
\end{example}
It is not clear whether it is possible to generalize this counterexample to show the following:
\begin{question}
    Is there a family of algebras $\{A'_n\}_{n\geq 1}$ such that $\inf_{n \geq 1}s(A'_n) = -\infty$?
\end{question}

After delooping levels were introduced, Guo and Igusa generalized this idea in \cite{guo_derived_2025}, creating more invariants that would bound the big finitistic dimension of the opposite algebra. In particular, they introduce the $k$-delooping level $k$-$\operatorname{dell}$, and the derived delooping level $\operatorname{ddell}$. 
\begin{definition}
     Let $k$ be a positive integer. The \textbf{$k$-delooping level} of $M$ is the least integer $n$ such that there exists an $A$-module $N$ with $\Omega^n(M)$ is a stable summand of $\Omega^{n+k}(N)$. We write $k\operatorname{-dell}_A(M) = n$. 
\end{definition}
\begin{definition}
    The \textbf{derived delooping level} of an $A$-module $M$, denoted by $\operatorname{ddell}(M)$, is:
    \begin{align*}
\operatorname{ddell}_A(M)& =\inf \{m \in \mathbb{N} \mid \exists n \leq m \text{ and an exact sequence in } \bmod A  \text{ of the form } \\
& 0 \rightarrow C_n \rightarrow C_{n-1} \rightarrow \cdots \rightarrow C_1 \rightarrow C_0 \rightarrow M \rightarrow 0 \\&\text { where } (i+1) \operatorname{-dell} C_i \leq m-i, i=0,1, \ldots, n\}
\end{align*}
\end{definition}
Similarly, we define for an algebra $A$ its $k$-delooping and derived delooping level by: \[
k\operatorname{-dell}(A) := \sup_{M \text{ is a simple module}} k\operatorname{-dell}_A(M),
\]
\[\operatorname{ddell}(A) := \sup_{M \text{ is a simple module}} \operatorname{ddell}_A(M).\]
Guo and Igusa showed that for any positive integer $k$:
\[k\operatorname{-dell}(A) \geq \operatorname{dell}(A) \geq \operatorname{ddell}(A) \geq \operatorname{FinDim}(A^{op}).\]

The difference between those invariants has been investigated in previous works. For example, Kershaw and Rickard \cite{kershaw_finite_2024} had an example where $\operatorname{dell}(A) =\infty$ but $\operatorname{FinDim}(A^{op})=1$, and it was later shown by Guo and Igusa \cite{guo_derived_2025} that $\operatorname{ddell}(A) =1$ as well. Barrios, Lanzilotta and Mata later proposed a family of algebras with two positive integer parameters $A_{n, k}$ (with $n > k$) where $\operatorname{dell}(A_{n,k})$ only depends on $n$ and $\operatorname{ddell}(A_{n, k}) = \operatorname{FinDim}(A_{n, k}^{op})$ only depends on $k$. Very recently, Gao, Liu and Xu \cite{gao2026deriveddeloopinglevelsonepoint} found an algebra $B$ where $\operatorname{ddell}(B)-\operatorname{FinDim}(B^{op})=1$ and .

Similar to Theorem \ref{thm:dell_gen} we can also define for a finite-dimensional algebra $R$:
\[s'(R):= \operatorname{ddell}(R\otimes_{\mathbf{k}}R) -2\operatorname{ddell}(R)\]

Guo \cite[Corollary 4.3, 4.4]{guo_derived_2025} established $s'(R) = 0$, and other results regarding derived delooping level and tensor product if $R$ possessed some desired homological properties. One can ask whether $s'(R) = 0$ for all finite-dimensional algebras $R$. The author suspects it has a negative answer, where the example $B$ in \cite{gao2026deriveddeloopinglevelsonepoint} might give $s'(B)=1$. 

We will end with a question that is analogous to Theorem \ref{thm:dell_gen}.

\begin{question}
    Is there a family of algebras $\{B_n\}_{n\geq 1}$ such that $\sup_{n\geq 1} s'(B_n) =\infty$?
\end{question}

\printbibliography
\end{document}